\documentclass[12pt]{amsart}

\usepackage{amsthm,amssymb,color,hyperref,mathtools,multirow,tikz,xinttools}
\usepackage[margin=1in]{geometry}

\theoremstyle{definition}
\newtheorem{thm}{Theorem}[section]
\newtheorem{prop}[thm]{Proposition}

\newtheorem{example}[thm]{Example}
\newtheorem{conj}[thm]{Conjecture}
\newtheorem{question}[thm]{Question}

\DeclareMathOperator{\Av}{Av}
\DeclareMathOperator{\Des}{Des}
\DeclareMathOperator{\Desbot}{Desbot}
\DeclareMathOperator{\Destop}{Destop}
\DeclareMathOperator{\des}{des}
\DeclareMathOperator{\inv}{inv}
\DeclareMathOperator{\maj}{maj}

\DeclareMathOperator{\exc}{exc}
\DeclareMathOperator{\st}{st}
\DeclareMathOperator{\std}{std}

\newcommand{\eps}{\varepsilon}
\newcommand{\NN}{\mathbb{N}}
\newcommand{\RR}{\mathbb{R}}
\newcommand{\ZZ}{\mathbb{Z}}
\newcommand{\symm}{\mathfrak{S}}
\newcommand{\desWilf}{\overset{\des}{\equiv}}
\newcommand{\stWilf}{\overset{\st}{\equiv}}

\usetikzlibrary{patterns}
\providecommand{\textscale}{0.25}
\newcommand{\drawthegrid}[1]{%
\draw (0.01,0.01) grid (#1+0.99,#1+0.99);
}
\newcommand{\drawthepoints}[2]{%
\foreach[count=\x] \y in {#1}
\filldraw (\x,\y) circle (#2 pt);
}
\newcommand{\shadeboxes}[1]{
    \foreach \x/\y in {#1}
    \fill[pattern=north east lines, pattern color=black!75] (\x,\y) rectangle +(1,1);
}
\newcommand{\clpattern}[3][5]{%
    \pgfmathsetmacro\circlesize{#1+4}
    \drawthegrid{\xintNthElt{0}{\xintCSVtoList {#3}}}
    \drawthepoints{#3}{#1}
    \foreach \x in {#2}
    \draw (\x,\xintNthElt{\x}{\xintCSVtoList {#3}}) circle (\circlesize pt);
}
\newcommand{\textpattern}[3][]{%
    \raisebox{0.6ex}{
    \begin{tikzpicture}[scale=\textscale, baseline=(current bounding box.center)]
        \shadeboxes{#3}
        \clpattern[8]{#1}{#2}
    \end{tikzpicture}}\;
}

\begin{document}

\title[Descents and $\des$-Wilf Equivalence]{Descents and $\des$-Wilf Equivalence of Permutations Avoiding Certain Non-classical Patterns}

\author{Caden Bielawa}
\address[Caden Bielawa, Robert Davis, Daniel Greeson]{
	Department of Mathematics \\
	Michigan State University \\
	East Lansing, MI 48824}
\email[Caden Bielawa]{bielawac@msu.edu}

\author{Robert Davis}
\email[Robert Davis]{davisr@math.msu.edu}

\author{Daniel Greeson}
\email[Daniel Greeson]{greesond@msu.edu}

\author{Qinhan Zhou}\thanks{The fourth author was supported in part by the Michigan State University Discovering America exchange program.}
\address[Qinhan Zhou]{
	School of Mathematics and Statistics \\
	Xi'an Jiaotong University \\
	Xi'an, Shaanxi 710049, China}
\email{qinhan\_zhou@qq.com}

\begin{abstract}
	A frequent topic in the study of pattern avoidance is identifying when two sets of patterns $\Pi, \Pi'$ are Wilf equivalent, that is, when $|\text{Av}_n(\Pi)| = |\text{Av}_n(\Pi')|$ for all $n$.
	In recent work of Dokos et al. the notion of Wilf equivalence was refined to reflect when avoidance of classical patterns preserves certain statistics.
	In this article, we continue their work by examining $\text{des}$-Wilf equivalence when avoiding certain non-classical patterns.
\end{abstract}

\maketitle


\section{Introduction}

Let $\symm_n$ denote the set of permutations of $[n] := \{1,\dots,n\}$, and let $\symm = \symm_1 \cup \symm_2 \cup \dots$ be the set of all permutations of finite length.
We write $\sigma \in \symm_n$ as $\sigma = a_1a_2\dots a_n$ to indicate that $\sigma(i) = a_i$.
A function $\st: \symm_n \to \NN$ is called a \emph{statistic}, and the systematic study of permutation statistics is generally accepted to have begun with MacMahon \cite[Volume I, Section III, Chapter V]{MacMahonCA}. 
Four of the most well-known statistics are the \emph{descent}, \emph{inversion}, \emph{major}, and \emph{excedance} statistics, defined respectively by
\begin{eqnarray*}
	\des(\sigma) &=& | \Des(\sigma)| \\
	\inv(\sigma) &=& | \{(i,j) \in [n]^2\ |\ i < j \text{ and } a_i > a_j \}| \\
	\maj(\sigma) &=& \sum_{i \in \Des(\sigma)} i \\
	\exc(\sigma) &=& |\{ i \in [n]\ |\ a_i > i\}|,
\end{eqnarray*}
where $\Des(\sigma) = \{i \in [n-1]\ |\ a_i > a_{i+1}\}$.
Given any statistic $\st$, one may form the generating function
\[
	F_n^{\st}(q) = \sum_{\sigma \in \symm_n} q^{\st \sigma}.
\]
A famous result due to MacMahon \cite{MacMahonCA} states that $F_n^{\des}(q) = F_n^{\exc}(q)$, and that both are equal to the \emph{Eulerian polynomial} $A_n(q)$.
Similarly, it is known that $F_n^{\inv}(q) = F_n^{\maj}(q) = [n]_q!$, where
\[
	[n]_q = 1 + q + \dots + q^{n-1} \text{ and } [n]_q! = [n]_q[n-1]_q\dots [1]_q.
\]

Let $A \subseteq [n]$, and denote by $\symm_A$ the set of permutations of the elements of $A$.
The \emph{standardization} of $\sigma = a_1\dots a_{|A|} \in \symm_A$ is the element of $\symm_{|A|}$ whose letters are in the same relative order as those of $\sigma$; we denote this permutation by $\std(\sigma)$.
Now, we say that a permutation $\sigma \in \symm_n$ \emph{contains the pattern} $\pi \in \symm_k$ if there exists a subsequence $\sigma' = a_{i_1}\dots a_{i_k}$ of $\sigma$ such that $\std(\sigma') = \pi$.
If no such subsequence exists, then we say that $\sigma$ \emph{avoids the pattern} $\pi$.
Since we will introduce additional notions of patterns, we may call such a pattern a \emph{classical pattern} to avoid confusion.
If $\Pi \subseteq \symm$, then we say $\sigma$ \emph{avoids} $\Pi$ if $\sigma$ avoids every element of $\Pi$.
The set of all permutations of $\symm_n$ avoiding $\Pi$ is denoted $\Av_n(\Pi)$.
In a mild abuse of notation, if $\Pi = \{\pi\}$, we will write $\Av_n(\pi)$.
If $\Pi,\Pi'$ are two sets of patterns and $|\Av_n(\Pi)| = |\Av_n(\Pi')|$ for all $n$, then we say $\Pi$ and $\Pi'$ are \emph{Wilf equivalent} and write $\Pi \equiv \Pi'$.

These ideas may be combined by setting
\[
	F_n^{\st}(\Pi;q) = \sum_{\sigma \in \Av_n(\Pi)} q^{\st \sigma}.
\]
This allows one to say that $\Pi,\Pi'$ are \emph{$\st$-Wilf equivalent} if $F_n^{\st}(\Pi;q) = F_n^{\st}(\Pi';q)$ for all $n$, and write this as $\Pi \stWilf \Pi'$.
Thus, $\Pi$ and $\Pi'$ may be Wilf equivalent without being $\st$-Wilf equivalent.
As a concrete example, $123$ and $321$ are clearly not $\des$-Wilf equivalent, even though they are Wilf equivalent.
It is straightforward to check that $\st$-Wilf equivalence is indeed an equivalence relation on $\symm$.

Since it is generally a difficult question to determine whether two sets are nontrivially Wilf equivalent, one should not expect it to be any easier to determine $\st$-Wilf equivalence.
However, it is certainly possible to obtain some results; see \cite{DokosEtAl} for results regarding $F_n^{\inv}$ and $F_n^{\maj}$, and \cite{Baxter,CameronKillpatrick} for further results, including a study of enumeration strategies for questions of this nature.
In this article, we will study $F_n^{\des}(\Pi;q)$ for certain non-classical patterns, called mesh patterns and barred patterns.
Special cases will allow us to identify $\des$-Wilf equivalences.
We will also present several conjectural $\des$-Wilf equivalences and provide computational evidence for these. 

\subsection{Acknowledgements}

The authors would like to thank the anonymous referee for his or her thoughtful comments which significantly improved this work.


\section{Pattern Avoidance Background}

\subsection{Classical Patterns}

In order to work most efficiently, it is important to recognize that certain Wilf equivalences are almost immediate to establish.
For example, it is obvious that $|\Av_n(123)| = |\Av_n(321)|$, since $a_1\dots a_n \in \Av_n(123)$ if and only if $a_na_{n-1}\dots a_1 \in \Av_n(321)$.
This idea can be generalized significantly.

\begin{figure}
	\begin{tikzpicture}
		\draw[step=1cm,very thin] (1,1) grid (6,6);
			\foreach \i in {(1,3),(2,4),(3,2),(4,5),(5,1),(6,6)}
			{
				\fill \i circle [radius = 3pt];
			}
	\end{tikzpicture}
	\caption{The plot of $342516$.}\label{fig: plot}
\end{figure}
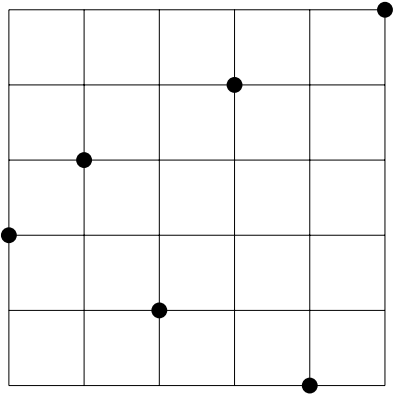

The $\emph{plot}$ of $\sigma \in \symm_n$ is the set of pairs $(i,\sigma(i)) \in \RR^2$ and will be denoted $P(\sigma)$. 
The plot of $342516$ is shown in Figure~\ref{fig: plot}.
Let $D_4 = \{R_0,R_{90},R_{180},R_{270},r_{-1},r_{0},r_{1},r_{\infty}\}$, where $R_{\theta}$ is counterclockwise rotation of a plot by an angle of $\theta$ degrees and $r_m$ is reflection across a line of slope $m$. 
A couple of these rigid motions have easy descriptions in terms of the one-line notation for permutations.
If $\pi = a_1a_2 \ldots a_k$ then its {\em reversal} is $\pi^r =a_k\ldots a_2a_1 =r_{\infty}(\pi)$, and its {\em complement} is $\pi^c = (k+1-a_1)\ (k+1-a_2)\ \dots \ (k+1-a_k)=r_0(\pi)$.

Note that $\sigma \in \Av_n(\pi)$ if and only if $f(\sigma) \in \Av_n(f(\pi))$ for any $f \in D_4$, hence $\pi \equiv f(\pi)$.
For this reason, the equivalences induced by the dihedral action on a square are often referred to as the {\em trivial Wilf equivalences}. 

Using these techniques, it is easy to show that $123$ and $321$ are trivially Wilf equivalent, as are all of $132$, $213$, $231$, and $312$.
It is less obvious, however, whether $123$ and $132$ are Wilf equivalent.
This question was settled by independent results due to MacMahon \cite{MacMahonCA} and Knuth \cite{KnuthVol1}, whose combined work showed that $\Av_n(132)$ and $\Av_n(123)$ are enumerated by the \emph{$n^{th}$ Catalan number}
\[
	C_n = \frac{1}{n+1}\binom{2n}{n}.
\]
The Catalan numbers are famous for appearing in a multitude of combinatorial situations; see \cite{StanleyCatalan} for many of them.

One of the most well-known combinatorial objects enumerated by the Catalan numbers are Dyck paths.
A \emph{Dyck path of length $2n$} is a lattice path in $\RR^2$ starting at $(0,0)$ and ending at $(2n,0)$, using steps $(1,1)$ and $(1,-1)$, which never goes below the $x$-axis.
See Figure~\ref{fig: dyck} for an example Dyck path of length $8$.

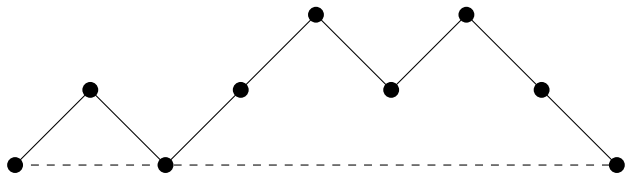
\begin{figure}
	\begin{tikzpicture}[every node/.style={inner sep=0}]
		\draw[dashed] (0,0) -- (8,0);

		\node (A) at (0,0) {};
		\node (B) at (1,1) {};
		\node (C) at (2,0) {};
		\node (D) at (3,1) {};
		\node (E) at (4,2) {};
		\node (F) at (5,1) {};
		\node (G) at (6,2) {};
		\node (H) at (7,1) {};
		\node (I) at (8,0) {};
		
		\draw (A) -- (B) -- (C) -- (D) -- (E) -- (F) -- (G) -- (H) -- (I);
		
		\foreach \xy in {(0,0),(1,1),(2,0),(3,1),(4,2),(5,1),(6,2),(7,1),(8,0)}
		{
			\fill \xy circle [radius=3pt];
		}
	\end{tikzpicture}
	\caption{A Dyck path of length $8$.}\label{fig: dyck}
\end{figure}


\subsection{Non-classical Patterns}

In this section, we will define two classes of non-classical patterns and describe what it means for a permutation to contain or avoid them.
The definitions of Wilf equivalence and $\des$-Wilf equivalence then extend to these patterns in the same way as classical patterns, so their precise definitions will be omitted.

A \emph{mesh pattern} is a pair $(\pi,M)$ where $\pi \in \symm_k$ and $M \subseteq [0,k]^2$.
Mesh patterns are a vast generalization of classical patterns and were first introduced by Br\"and\'en and Claesson \cite{BrandenClaessonMesh}.
It is convenient to represent a mesh pattern as a grid which plots $\pi$ and shades in the unit squares whose bottom-left corners are the elements of $M$.
For example, one may represent the mesh pattern $(\pi_0,M_0) = (4213,\{(0,2),(1,0),(1,1),(3,3),(3,4),(4,3)\})$ as follows:
\[
	(\pi_0,M_0) = \textpattern{4,2,1,3}{0/2,1/0,1/1,3/3,3/4,4/3}
\]

Containment of mesh patterns is most easily understood by an informal statement but illustrative examples; the formal definition, given in \cite{BrandenClaessonMesh}, shows that the intuition developed this way behaves as expected. 
We say that $\sigma \in \symm_n$ \emph{contains the mesh pattern} $(\pi,M)$ if $\sigma$ contains an occurrence of $\pi$ and the shaded regions of $P(\pi)$ corresponding to this occurrence contain no other elements of $P(\sigma)$.
If $\sigma$ does not contain $(\pi,M)$, then we say $\sigma$ \emph{avoids} $(\pi,M)$.

For the illustrative examples, first consider $\sigma = 612435$.
Notice that while $6435$ is an occurrence of $4213$ in $\sigma$, it is not an occurrence of the mesh pattern $(\pi_0,M_0)$ given above, since the shaded regions in $P(\sigma)$ dictated by $M_0$ yield
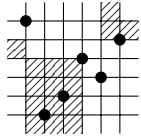
\[
	\textpattern{6,1,2,4,3,5}{0/4,1/0,1/1,1/2,1/3,2/0,2/1,2/2,2/3,3/0,3/1,3/2,3/3,5/5,5/6,6/5}
\]
Now consider $\sigma' = 153624$.
In this case, $5324$ is an occurrence of both $4213$ and $(\pi_0,M_0)$ in $\sigma'$, since the shading in this case is
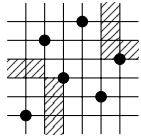
\[
	\textpattern{1,5,3,6,2,4}{0/3,1/3,2/0,2/1,2/2,5/4,5/5,5/6,6/4}
\]

In certain cases, determining which permutations avoid a mesh pattern $(\pi,M)$ with $M$ nonempty is equivalent to determining which permutations avoid $\pi$ as a classical pattern.
When this happens, we say that $(\pi,M)$ has \emph{superfluous mesh}, and Tenner \cite{TennerSuperfluous} identified when exactly a mesh pattern has superfluous mesh.
To do this, we first define an \emph{enclosed diagonal} of $(\pi,M)$ to be a triple $((i,j),\varepsilon,\ell)$ where $\varepsilon \in \{-1,1\}$, $\ell \geq 1$, and the following three properties hold:
\begin{enumerate}
	\item The plot of $\pi$ contains the set $\{(i+d,j+\varepsilon d) \mid 1 \leq d < \ell\}$;
	\item The plot of $\pi$ contains neither $(i,j)$ nor $(i+\ell,j+\varepsilon \ell)$;
	\item $\{(i+d,j+\varepsilon d) \mid 0 \leq d < \ell\} \subseteq M$.
\end{enumerate}

Note that an enclosed diagonal may consist of a single element, as long as the corresponding box in the mesh pattern contains no element of $P(\pi)$.
To illustrate, the following three mesh patterns all have a unique enclosed diagonal:
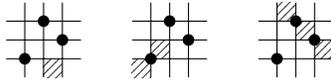
\[
	\textpattern{1,3,2}{2/0} \quad \textpattern{1,3,2}{0/0,1/1} \quad \textpattern{1,3,2}{1/3,2/2,3/1}
\]
However, none of the following five mesh patterns have any enclosed diagonals:
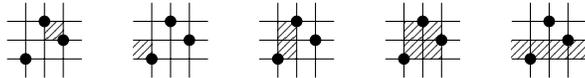
\[
	\textpattern{1,3,2}{2/2} \quad \textpattern{1,3,2}{0/1} \quad \textpattern{1,3,2}{1/1,1/2} \quad \textpattern{1,3,2}{1/1,1/2,2/1,2/2} \quad \textpattern{1,3,2}{0/1,1/1,2/1,3/1}
\]
The following theorem gives the characterization of when a pattern has superfluous mesh.
As a result, we will not focus on any patterns with superfluous mesh, but we will still use the theorem briefly.

\begin{thm}[{\cite[Theorem~$3.5^{\prime}$]{TennerSuperfluous}}]\label{thm:superfluous}
	A mesh pattern has superfluous mesh if and only if it has no enclosed diagonals.
\end{thm}

Mesh patterns also generalize \emph{$1$-barred patterns}, in which a classical pattern is allowed (but not required) to have a bar above one letter.
This is a special case of \emph{barred patterns}, in which each letter is allowed to have a bar above it.
The bars above letters indicate that certain additional rules are required in order to define containment of the pattern.
We will not give the precise definition of containment and avoidance of barred patterns in general, but will observe that if there are two or more bars in the pattern, there is not necessarily a simple translation of the barred pattern into a mesh pattern. 
In some instances, a barred pattern may be described as a \emph{decorated mesh pattern} \cite{UlfarssonDecoratedMesh}, but this is not always possible.
To avoid this difficulty in the statement and proof of Proposition~\ref{prop:barred prop}, we will simply describe here what it means for a permutation to avoid two specific barred patterns.

We say that $\sigma = a_1\dots a_n$ \emph{avoids} $\bar1\bar243$ if, whenever $a_ia_j$ is an occurrence of $21$, then there are some integers $k,l$ such that $k < l < i$ and $a_ka_la_ia_j$ is an occurrence of $1243$. 
We also say that $\sigma$ \emph{avoids} $\bar132\bar4$ if, whenever $a_ia_j$ is an occurrence of $21$, then there are some integers $k,l$ such that $k < i < j < l$ and $a_ka_ia_ja_l$ is an occurrence of $1324$.
As an example, $\sigma = 124635$ avoids $\bar1\bar243$ since all occurrences of $21$, which are $43$, $63$, and $65$, extend to an occurrence of $1243$ by placing $12$ before them.
However, $\sigma$ contains $\bar132\bar4$ since $63$, which is an occurrence of $21$, does not play the role of $32$ in any occurrence of $1324$ in $\sigma$. 

\section{Main Results}

We now have all of the tools we need to begin proving results.
We begin with a simple application of several known theorems.

\begin{prop}
	If $(132,M_1)$ and $(312,M_2)$ are mesh patterns, neither of which contain an enclosed diagonal, then
	\[
		(132,M_1) \desWilf (312,M_2).
	\]
\end{prop}

\begin{proof}
	By Theorem~\ref{thm:superfluous}, $\Av_n((312,M_2)) = \Av_n(312)$, so $(312,M_2) \desWilf 312$.
	It then follows directly from \cite[Remark~2.5(b)]{Reifegerste} that the number of elements in $\Av_n(312)$ with exactly $k$ descents is
	\[
		N_{n,k} := \frac{1}{n}\binom{n}{k}\binom{n}{k+1}.
	\]
	Since the sequence $\{N_{n,k}\}_{k = 0}^{n-1}$ is symmetric for fixed $n$, and since
	\[
		\des(\sigma) = n - 1 - \des(\sigma^c),
	\]
	we have 
	\[
		(312,M_2) \desWilf 312 \desWilf 132.
	\]
	Again by Theorem~\ref{thm:superfluous}, we have $\Av_n(132) = \Av_n((132,M_1))$, so these two patterns are $\des$-Wilf equivalent as well.
	Connecting the equivalences, the claim follows.
\end{proof}


Characterizing the $\des$-Wilf classes for mesh patterns $(\pi,M)$ where $\pi \in \symm_4$ is difficult, and we will not attempt to fully characterize the $\des$-Wilf equivalence classes of such patterns.
In what follows, we merely wish to present a step toward understanding these in more depth, but first we need two more definitions.

If $A \subseteq [n]$, $f \in D_4$, and $\sigma \in \symm_A$, then we let $f^A(\sigma)$ denote the unique element of $\symm_A$ whose standardization is $f(\std(\sigma))$.
We say that $f^A$ is a dihedral action \emph{relative to $A$}.
As a simple example, if $7461 \in \symm_{\{1,4,6,7\}}$, then $\std(7461) = 4231$ and $R_{90}^{\{1,4,6,7\}}(\sigma) = 1647$. 

\begin{thm} We have
\[
	\textpattern{3,1,2,4}{1/0,1/1,1/2,1/3,1/4,2/0,2/1,2/2,2/3,2/4}
	\desWilf
	\textpattern{2,3,1,4}{1/0,1/1,1/2,1/3,1/4,2/0,2/1,2/2,2/3,2/4}
	\desWilf
	\textpattern{2,4,1,3}{1/0,1/1,1/2,1/3,1/4,2/0,2/1,2/2,2/3,2/4}
\]
\end{thm}

\begin{proof}
	First consider $(\pi_1,M_1) = \textpattern{3,1,2,4}{1/0,1/1,1/2,1/3,1/4,2/0,2/1,2/2,2/3,2/4}$ and $(\pi_2,M_2) = \textpattern{2,3,1,4}{1/0,1/1,1/2,1/3,1/4,2/0,2/1,2/2,2/3,2/4}$. 
	To prove their $\des$-Wilf equivalence, we will form a $\des$-preserving bijection
	\[
		\alpha: \symm_n \setminus \Av_n((\pi_1,M_1)) \to \symm_n \setminus \Av_n((\pi_2,M_2)),
	\]
	that is, a $\des$-preserving bijection between permutations in $\symm_n$ containing $(\pi_1,M_1)$ and those containing $(\pi_2,M_2)$.

	Suppose $\sigma = a_1\dots a_n \in \symm_n$ contains $(\pi_1,M_1)$.
	If $\sigma$ contains $(\pi_2,M_2)$, then set $\alpha(\sigma) = \sigma$.
	Otherwise, let $j$ be the smallest index in which an occurrence of $(\pi_1,M_1)$ begins, and consider $a_ia_{i+1}\dots a_p$, where
	\[
		 p = \min \{ m \mid m > j+2,\, a_m >a_j\} \text{ and } i = \min \{ m \mid m \leq j,\, a_m,a_{m+1},\dots,a_j < a_p\}.
	\]
	
	Let $A = \{a_i, a_{i+1}, \dots, a_p\}$, and set
	\[
		R_{180}^A(a_i\dots a_p) = b_i\dots b_p,
	\]	
	and further set 
	\[
		\alpha(\sigma) = a_1\dots a_{i-1}b_i\dots b_pa_{p+1}\dots a_n. 
	\]
	
	Since $R_{180}^A$ is aa $\des$-preserving map, we have that for any $k \in \{1,\dots, p-1-i\}$, $i + k \in \Des(\sigma)$ if and only if $p-k \in \Des(\alpha(\sigma))$.
	Additionally, for any $k \in \{1,\dots,i-1,p,p+1,\dots,n-1\}$, $k \in \Des(\sigma)$ if and only if $k \in \Des(\alpha(\sigma))$.	
	Thus, $\alpha$ is $\des$-preserving.

	To show that $\alpha$ is invertible, we will construct a map
	\[
		\beta: \symm_n \setminus \Av_n((\pi_2,M_2)) \to \symm_n \setminus \Av_n((\pi_1,M_1))
	\]
	and show that $\beta \circ \alpha$ is the identity map on $\symm_n \setminus \Av_n((\pi_1,M_1))$.
	If $\sigma' = a_1'\dots a_n'$ contains $(\pi_2,M_2)$, then we create $\beta(\sigma)$ by first testing a construction similar to the one from the previous paragraph.
	Namely, let $j'$ be the smallest index in which an occurrence of $(\pi_2,M_2)$ begins, and consider $a_i'a_{i+1}'\dots a_p'$, where
	\[
		 p' = \min \{ m \mid m > j'+2,\, a_m' >a_{j'+1}'\} \text{ and } i' = \min \{ m \mid m \leq j',\, a_m',a_{m+1}',\dots,a_{j'} < a_p'\}.
	\]
	This time, let $A' = \{a_i', a_{i+1}', \dots, a_p'\}$, and set
	\[
		R_{180}^{A'}(a_i'\dots a_p') = b_i'\dots b_p'.
	\]
	If $a_1'\dots a_{i-1}'b_i'\dots b_{p'}'a_{p'+1}'\dots a_n'$ contains both $(\pi_2,M_2)$ and $(\pi_1,M_1)$, then set $\beta(\sigma') = \sigma'$.
	Otherwise, set $\beta(\sigma') = a_1'\dots a_{i-1}'b_i'\dots b_{p'}'a_{p'+1}'\dots a_n'$.
	The fact that $\beta \circ \alpha$ is the identity map on $\symm_n \setminus \Av_n((\pi_1,M_1))$ follows from construction. 
	
	Now consider $(\pi_2,M_2)$ and $(\pi_3,M_3) = \textpattern{2,4,1,3}{1/0,1/1,1/2,1/3,1/4,2/0,2/1,2/2,2/3,2/4}$.
	Suppose $\sigma=a_1a_2\cdots a_n$ and $a_ja_{j+1}a_{j+2}a_p$ is the first copy of $(\pi_3,M_3)$, as identified in the second paragraph in this proof.
	If $a_p$ is the only $a_l$ for which $l > j+2$ and $a_l>a_j$, then set $\alpha(\sigma)$ to be $\sigma$ with $a_{j+1}$ and $a_p$ transposed.
	Otherwise, choose 
	\[
		r = \min \{ l \mid a_j<a_l<a_{j+1},\,l>j+2 \}.
	\]
	Let $S=\{ a_r,a_{r+1},\dots,a_q\}$ where $q$ is the maximum index for which $\{a_r,a_{r+1},\dots,a_q\}$ is increasing and $a_j < a_k < a_{j+1}$ for all $k \in S$.
        Set $\alpha(\sigma)$ to be $\sigma$ with $a_{j+1}$ and $\max S$ transposed.
        By choosing the maximum of $S$ we are guaranteeing that $\alpha$ is $\des$-preserving.
        By construction, $\alpha(\sigma)$ contains an occurrence of $(\pi_2,M_2)$.
        Using an argument similar to the first part of this proof, $\alpha$ is invertible and is therefore a bijection.
\end{proof}


%
%

Recall that the \emph{Stirling numbers of the second kind}, denoted $S(n,k)$, record the number of ways to partition $[n]$ into $k$ nonempty blocks.
Here, we will begin to find useful the notation
\[
	\Av_n^{\des,k}(\Pi) = \{ \sigma \in \Av_n(\Pi) \mid \des(\sigma) = k\}.
\]

\begin{prop}
	Let $(\pi,M) = \textpattern{1,3,2}{2/0}$. 
	For all $n$, we have
	\[
		F_n^{\des}\left((\pi,M); q\right) = \sum_{k = 0}^{n-1} S(n,k+1)q^k.
	\]
\end{prop}

\begin{proof}
	Let $\Sigma_{n,k}$ denote the collection of set partitions of $[n]$ into exactly $k$ nonempty blocks.
	We will create a bijection
	\[
		f: \Av_n^{\des,k}((\pi,M)) \to \Sigma_{n,k+1},
	\]
	from which the conclusion follows.
	
	First, let $\sigma = a_1\dots a_n \in \Av_n^{\des,k}((\pi,M))$.
	It follows from \cite[Theorem 4.1]{BursteinLankham} that any such permutation is the concatenation of substrings
	\begin{eqnarray*}
		a_1 &< \dots <& a_{i_0} \\
		a_{i_0+1} &< \dots <& a_{i_1} \\
		 & \vdots & \\
		a_{i_k + 1} &< \dots <& a_n,
	\end{eqnarray*}
	where $a_1 < a_{i_j+1} > a_{i_{j+1}+1} $ for all $j$.
	In particular, the values $a_{i_0},\dots,a_{i_k}$ determine the entire permutation.
	
	Associate to $\sigma$ the set partition
	\[
		f(\sigma) = \{\{a_1, \dots,  a_{i_0}\}, \{a_{i_0+1}, \dots, a_{i_1}\}, \dots, \{a_{i_k + 1}, \dots, a_n\}\}.
	\]
	Note that if $\sigma = 12\dots n$, then $k = 0$, so this partition consists of only one block.
	Thus, if $\sigma$ has $k$ descents, then the partition obtained has $k+1$ blocks. 
	Because each choice of the $a_{i_j}$ determines $\sigma$, we know that $f(\sigma) \neq f(\sigma')$ whenever $\sigma' \in \Av_n^{\des,k}((\pi,M))$ and $\sigma \neq \sigma'$. 
	That is, $f$ is injective.
	
	Now we will show that $f$ is surjective.
	Consider a set partition $B = \{B_1,\dots,B_{k+1}\}$ of $[n]$ into $k+1$ blocks.
	We are free to write the $B_i$ such that
	\[
		B_i = \{b_{i,1} < \dots < b_{i,i_l}\} \text{ and } \min B_i < \min B_{i+1}
	\]
	for all $i$.
	Construct the permutation 
	\[
		b_{k+1,1}b_{k+1,2}\dots b_{k+1,i_{k+1}}b_{k,1}b_{k,2}\dots b_{k,i_k}\dots b_{1,1}b_{1,2}\dots b_{1,i_1}.
	\]
	We claim that this permutation is an element of $\Av_n^{\des,k}((\pi,M))$.
	
	Any occurrence of $\textpattern{1,3,2}{}$, say, $b_{\alpha}b_{\beta}b_{\gamma}$, implies that $b_{\alpha} \in B_i$, $b_{\beta} \in B_j$, and $b_{\gamma} \in B_k$ for some $i \leq j < k$. 
	Since the sequence of minima of the blocks is decreasing, we know that $\min B_k < b_{\alpha} < b_{\gamma}$.
	Thus, the string
	\[
		b_{\alpha}b_{\beta}(\min B_k)b_{\gamma}
	\]
	is an occurrence of $\textpattern{1,3,2}{2/0}$.
	Since the elements of the blocks strictly increase, the minima decrease, and since there are $k+1$ blocks, there are $k$ descents in the permutation.
	Thus $f$ is surjective, completing the proof.
\end{proof}

\begin{example}
	Consider the permutation $3427156 \in \Av_6^{\des,2}\left(\textpattern{1,3,2}{2/0}\right)$.
	Our construction in the previous proof associates to this permutation the partition
	\[
		\{\{3,4\},\{2,7\},\{1,5,6\}\}.
	\]
	In the other direction, given the set partition
	\[
		\{\{5\},\{3,1,4\},\{7,2,6\}\} = \{\{5\},\{2,6,7\},\{1,3,4\}\},
	\]
	we obtain the permutation $5267134$, which the reader may verify is indeed an element of $\Av_7^{\des,2}\left(\textpattern{1,3,2}{2/0}\right)$. 
\end{example}

A \emph{Motzkin path of length $n$} is a lattice path from $(0,0)$ to $(n,0)$ using only \emph{up-steps} $(1,1)$, \emph{down-steps} $(1,-1)$, and \emph{horizontal steps} $(1,0)$ such that the path does not go below the $x$-axis. 
An example is shown in Figure~\ref{fig: Motzkin path}.
We let $\mathcal{M}_{n,k}$ denote the set of Motzkin paths of length $n$ with exactly $k$ up-steps.

The next result we present was first proven in \cite{ChenDengYang} by writing Motzkin paths according to a ``strip decomposition'' and by writing permutations according to canonical reduced decompositions.
Here, we present a new, simpler proof.
To do so, we only need a few more definitions.

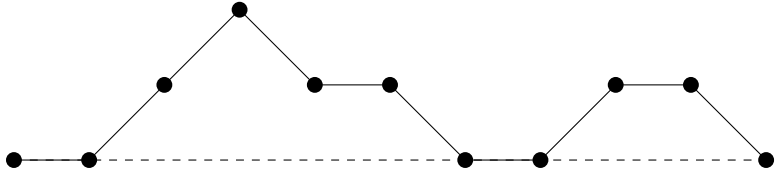
\begin{figure}
	\begin{tikzpicture}[every node/.style={inner sep=0}]
		\draw[dashed] (0,0) -- (10,0);
		\node (A) at (0,0) {};
		\node (B) at (1,0) {};
		\node (C) at (2,1) {};
		\node (D) at (3,2) {};
		\node (E) at (4,1) {};
		\node (F) at (5,1) {};
		\node (G) at (6,0) {};
		\node (H) at (7,0) {};
		\node (I) at (8,1) {};
		\node (J) at (9,1) {};
		\node (K) at (10,0) {};
		
		\draw (A) -- (B) -- (C) -- (D) -- (E) -- (F) -- (G) -- (H) -- (I) -- (J) -- (K);
		
		\foreach \xy in {(0,0),(1,0),(2,1),(3,2),(4,1),(5,1),(6,0),(7,0),(8,1),(9,1),(10,0)}
		{
			\fill \xy circle [radius=3pt];
		}
	\end{tikzpicture}
	\caption{A Motzkin path of length $10$ with $3$ up-steps.}\label{fig: Motzkin path}
\end{figure}

If $i$ is a descent of $\sigma = a_1\dots a_n$, then we call $a_i$ a \emph{descent top} and $a_{i+1}$ a \emph{descent bottom}.
Let $\Destop(\sigma)$ denote the set of descent tops of $\sigma$ and let $\Desbot(\sigma)$ denote the set of descent tops of $\sigma$. 
A \emph{valley} in $\sigma$ is an element $i$ for which $a_{i-1} > a_i < a_{i+1}$.

\begin{thm}[Theorem 3.1, \cite{ChenDengYang}]
	Let $\Pi = \left\{\textpattern{3,2,1}{}, \textpattern{2,3,1}{1/0}\right\}$.
	For all $n$,
	\[
		F_n^{\des}(\Pi; q) = \sum_{k=0}^n |\mathcal{M}_{n,k}|q^k
	\]
\end{thm}

\begin{proof}
	We will form a bijection
	\[
		\mu: \Av_n^{\des,k}\left(\Pi\right) \to \mathcal{M}_{n,k}.
	\]
	For $\sigma = a_1\dots a_n \in \Av_n^{\des,k}\left(\Pi\right)$, let $\mu(\sigma)$ be the lattice path obtained by making step $a_i$ a down-step if $a_i$ is a descent bottom, 
	an up-step if $a_i$ is a descent top, and a horizontal step if $a_i$ is neither.
	
	First, we need to check that $\mu$ is well-defined.
	Note that no letter of $\sigma$ can be both a descent top and a descent bottom, since this would imply $\sigma$ contains an instance of $\textpattern{3,2,1}{}$, which is forbidden.
	So, since the sets of descent tops and of descent bottoms are disjoint, and these appear in pairs, then we can be certain that the path constructed by $\mu$ has length $n$ and ends at $(n,0)$.
	Moreover, since a descent top always appears before a descent bottom, at no step of the path can there have been more down-steps than up-steps.
	This establishes that $\mu(\sigma)$ is a Motzkin path of length $n$.
	Finally, since there are $k$ descents, there are $k$ descent tops, and $\mu(\sigma)$ will have $k$ up-steps.
	Hence, $\mu(\sigma) \in \mathcal{M}_{n,k}$.
	 
	Next we will show that $\mu$ is injective.
	To do so, we will determine exactly the structure of the elements in $\Av_n\left(\Pi\right)$.
	Notice that the descent bottoms of $\sigma$ must appear in increasing order in $\sigma$, since, otherwise, there would be an occurrence of $\textpattern{3,2,1}{}$.
	For the same reason, the descent tops must appear in increasing order in $\sigma$.
	
	Let $\sigma = a_1\dots a_n \in \Av_n^{\des,k}\left(\Pi\right)$ and suppose that $i$ is neither a descent top nor descent bottom.
	Suppose for now that $j$ is the first descent greater than $i$.
	If $a_{j+1} < a_i < a_j$, then $a_ia_ja_{j+1}$ is an occurrence of $231$.
	Since $\sigma$ avoids $\textpattern{2,3,1}{1/0}$, there must be some $l$ for which $\sigma$ has the subsequence $a_ia_la_ja_{j+1}$ and $a_l < a_{j+1}$.
	This implies that some integer $i+1,i+2,\dots,l-1$ is a descent, which contradicts the fact that $j$ is the first descent greater than $i$. 
	So, it must be true that $a_i < a_{j+1} < a_j$.
	Since $j$ is the first descent greater than $i$, it follows that $a_ia_{i+1}\dots a_{j-1}a_{j+1}$ is an increasing sequence.
	It follows that the subsequence of $\sigma$ consisting of all letters that are not descent tops is an increasing sequence.
	
	Now we will show that $\mu$ is injective.
	If $\mu(\sigma_1) = \mu(\sigma_2)$ for $\sigma_1,\sigma_2 \in \Av_n\left(\Pi\right)$, then $\sigma_1$ and $\sigma_2$ have the same descent topsets and the same descent bottomsets, since these are identified by the up-steps and down-steps in the Motzkin path.
	Our description of elements of $\Av_n\left(\Pi\right)$ shows that once the descent topsets and descent bottomsets have been identified, there is a unique $\sigma$ in the avoidance class with those sets.
	Therefore, $\sigma_1 = \sigma_2$, and $\mu$ is injective. 
	
	Finally, we will show that $\mu$ is surjective.
	Let $A \in \mathcal{M}_{n,k}$, and label its steps $1,\dots,n$ from left to right.
	We will construct its preimage in stages.
	First write down $1,\dots,n$ but excluding the labels on the down-steps.
	Then insert the label on the $i^{th}$ down-step immediately before the label of the $i^{th}$ up-step.
	Call the resulting permutation $\sigma_A$.
	Using the description of elements of $\Av_n\left(\Pi\right)$ from earlier in this proof, we see that $\sigma_A \in \Av_n\left(\Pi\right)$.
	Additionally, it is clear that $\mu(\sigma_A) = A$ by our construction of $\sigma_A$ and the definition of $\mu$.
	Therefore, $\mu$ is surjective, completing the proof.
\end{proof}

\begin{example}
	Let $A$ be the Motzkin path in Figure~\ref{fig: Motzkin path}.
	Steps $2$, $3$, and $8$ are up-steps, and therefore will be descents bottoms.
	Steps $4$, $6$, and $10$ are down-steps, so these will be descent tops. 
	The remaining numbers will be neither descent tops nor bottoms.
	
	When the descent tops are removed from $\mu^{-1}(A)$, the result will be an increasing string of numbers: $1235789$. 
	The descent tops are then placed immediately preceding the descent bottoms, to obtain $1426357(10)89$.
\end{example}

For the final result of the section, we make two notes.
First, recall that the \emph{Eulerian polynomial} $A_n(q)$ is the polynomial
\[
	\sum_{\sigma \in \symm_n} q^{\des(\sigma)} = A_n(q).
\]
It should be noted that some authors (e.g. in \cite{StanleyVol1}) define the Eulerian polynomials using $q^{\des(\sigma) + 1}$ rather than the definition given here.
So, one should take care when encountering Eulerian polynomials in the literature.
Second, recall from the end of Section~$2$ what it means for a permutation to contain and avoid the barred patterns $\bar1\bar243$ or $\bar132\bar4$.

\begin{prop}\label{prop:barred prop}
	For all $n$,
	\[
		F_n^{\des}(\bar1\bar243; q) = F_n^{\des}(\bar132\bar4; q) = 
			\begin{cases}
				1 & \text{ if } n = 0,1 \\
				A_{n-2}(q) & \text{ if } n \geq 2 
			\end{cases}
	\]
\end{prop}

\begin{proof}
	We will first show that $F_n(\bar1\bar243;q)$ satisfies the right hand side.
	The conclusion is clearly true for $n < 2$, so we will restrict our attention to when $n \geq 2$.
	Choose $\sigma = a_1\dots a_n \in \Av_n(\bar1\bar243)$.
	Note first that $a_1 < a_2$ since, if $a_1 > a_2$, then $a_1a_2$ would be an occurrence of $u(\bar1\bar243) = 21$ but this cannot extend to an occurrence of $1243$.

	Now, suppose $a_2 > 2$.
	Setting $a_m = \min\{a_i \mid 3 \leq i \leq n \}$we have $a_2 > a_m$, so $a_2a_m$ is an occurrence of $u(\bar1\bar243)$ in $\sigma$.
	However, there is only letter to the left of $a_2$, so this pattern does not extend to an instance of $1243$.
	Thus, $a_2 = 2$.
	Together with the previous paragraph, we know $a_1 = 1$ as well.
	In particular, $a_1 < a_2 < a_i$ for all $i \geq 3$.
	
	Now, take any occurrence $a_ia_j$ of $21$ in which $2 < i < j$.
	Clearly, $a_1a_2a_ia_j$ is an extension to $1243$.
	This holds for any possible permutation of $3,\dots,n$ as the final $n-2$ letters.
	Since $1$ and $2$ are never descents of these permutations, we have
	\[
		F_n^{\des}(\bar1\bar243;q) = A_{n-2}(q),
	\]
	as claimed.
	
	Now we will show that the same formula holds for $\bar132\bar4$. 
	This time, assume $\sigma \in \Av_n(\bar132\bar4)$.
	If $a_i = 1$ for some $i > 1$, then $a_1a_i$ would be an occurrence of $21$.
	However, this can never extend to $1324$ since there is no letter to the left of $a_1$.
	Thus, $a_1 = 1$.
	An analogous argument shows $a_n = n$. 
	
	This allows $a_2\dots a_{n-1}$ to be any arrangement of $2,3,\dots,n-1$, since, whenever $a_ia_j$ is an occurrence of $21$ for $2 \leq i,j \leq n-1$, this extends to $1a_ia_jn$.
	So, we have the bijection
	\[
		a_1a_2\dots a_n \mapsto (a_2-1)(a_3-1)\dots(a_{n-1}-1)
	\]
	with elements of $\symm_{n-2}$.
	Since $1$ and $n$ are never descents in $\Av_n(\bar132\bar4)$, this is a $\des$-preserving bijection.
	Therefore, $F_n^{\des}(\bar1\bar243;q) = A_{n-2}(q)$.	
\end{proof}


\section{Conjectures and Further Directions}

In this section, we provide a few conjectures, supporting data, and additional direction in which this work could proceed.
In all cases, no closed forms for the functions $F_n^{\des}(\Pi;q)$ are known.
We refer the reader to Table~\ref{tab:data} for all known polynomials $F_n^{\des}(\Pi;q)$ for $4 \leq n \leq 8$, since, for these choices of $\Pi$, $F_n^{\des}(\Pi;q) = F_n^{\des}(\emptyset;q)$ for $n \leq 3$.

\begin{conj}
The following $\des$-Wilf equivalences hold:
\[
	\textpattern{1,2,4,3}{1/0,1/1,1/2,1/3,1/4}
	\desWilf
	\textpattern{3,4,1,2}{1/0,1/1,1/2,1/3,1/4}
	\quad\text{and}\quad
	\textpattern{1,3,4,2}{1/0,1/1,1/2,1/3,1/4}
	\desWilf
	\textpattern{2,4,1,3}{1/0,1/1,1/2,1/3,1/4}.
\]
\end{conj}

To state our next conjecture, we must discuss a particular sorting of permutations.
Let $\sigma = a_1\dots a_n \in \symm_n$ and suppose $a_i = n$. 
Let $\Gamma$ be the operator defined recursively as
\[
	\Gamma(\sigma) = \Gamma(a_1\dots a_{i-1})\Gamma(a_{i+1}\dots a_n)n.
\]
We say that $\sigma$ is \emph{West-$t$-stack-sortable} if $\Gamma^t(\sigma)$ is the identity permutation.
Note that the $2$-West-stack-sortable permutations \cite{WestThesis} are exactly those in
\[
	\Av_n\left(\textpattern{2,3,4,1}{},\textpattern{3,2,4,1}{1/4}\right).
\]

\begin{conj}
	The following $\des$-Wilf equivalence holds:	
\[
	\left\{
	\textpattern{2,3,4,1}{},
	\textpattern{3,2,4,1}{1/4}
	\right\}
	\desWilf
	\left\{
	\textpattern{2,4,1,3}{},
	\textpattern{3,1,4,2}{2/2}
	\right\}
\]
\end{conj}

If this conjecture is true, then from \cite{Bona02} it follows that
\begin{eqnarray*}
	F_n^{\des}\left(\left\{
	\textpattern{2,3,4,1}{},
	\textpattern{3,2,4,1}{1/4}
	\right\};q\right) &=& F_n^{\des}\left(\left\{
	\textpattern{2,4,1,3}{},
	\textpattern{3,1,4,2}{2/2}
	\right\};q\right) \\
		&=& \sum_{k=0}^{n-1} \frac{(n+k)!(2n-k-1)!}{(k+1)!(n - k)!(2k + 1)!(2n - 2k - 1)!}q^k.
\end{eqnarray*}

\begin{table}
\[
	\begin{array}{|c|c|c|}\hline \relax
		\Pi & n & F_n^{\des}(\Pi;q) \\ \hline \relax
		\multirow{2}{*}{$\left\{\textpattern{1,2,4,3}{1/0,1/1,1/2,1/3,1/4}\right\},\left\{\textpattern{3,4,1,2}{1/0,1/1,1/2,1/3,1/4}\right\}$} & 4 &  1+10q+ 11q^2 +q^3 \\ 
		 & 5 & 1+ 20q+ 57q^2 +26q^3 +q^4 \\
		 & 6 & 1+ 35q + 204q^2 + 252q^3 + 57q^4 + q^5 \\ 
		 & 7 & 1+ 56q+ 581q^2+ 1500q^3+ 969q^4+ 120q^5+ q^6 \\
		 & 8 & 1+84q+ 1414q^2 + 6588q^3 + 9117q^4 + 3426q^5+ 247q^6+ q^7 \\ \hline \relax
		 \multirow{2}{*}{$\left\{\textpattern{1,3,4,2}{1/0,1/1,1/2,1/3,1/4}\right\},\left\{\textpattern{2,4,1,3}{1/0,1/1,1/2,1/3,1/4}\right\}$} & 4 & 1 + 10q+ 11q^2 + q^3 \\
		 & 5 & 1+ 20q+ 56q^2+ 26q^3+ q^4 \\
		 & 6 & 1+ 35q+ 196q^2+ 241q^3+ 57q^4+ q^5\\ 
		 & 7 & 1+ 56q+ 546q^2+ 1361q^3+ 897q^4+ 120q^5 +q^6 \\ 
		 & 8 & 1+ 84q+ 1302q^2+ 5675q^3+ 7739q^4+ 3060q^5+ 247q^6+ q^7 \\ \hline
	\end{array}
\]
\caption{The polynomials $F_n^{\des}(\Pi;q)$ for certain sets of patterns $\Pi$.}\label{tab:data}
\end{table}

Instead of generalizing the patterns being avoided, one may generalize permutations themselves.
One way to do this is to consider the \emph{colored permutations} 
\[
	G_{r,n} := \{(\eps,\sigma) \mid \eps \in \ZZ_r,\, \sigma \in \symm_n\}.
\]
In this case, we say that $(\eps,\sigma) \in G_{r,n}$ \emph{contains} $(\zeta, \pi) \in G_{s,m}$ if there are elements
$1 \leq i_1 < i_2 < \dots < i_s \leq n$ such that $\std(\sigma_{i_1}\dots\sigma_{i_s}) = \pi$ and $\eps_{i_j} = \zeta_j$ for all $j$.
If no such choice of $i_j$ exist, then we say $(\eps,\sigma)$ \emph{avoids} $(\zeta,\pi)$.
For a set of colored permutations $\Pi$, let
\[
	\Av_{r,n}(\Pi) = \{(\eps,\sigma) \in G_{r,n} \mid (\eps,\sigma) \text{ avoids all } (\zeta,\pi) \in \Pi\}.
\]
\begin{question}
	What can be said about the polynomials
	\[
		F_{r,n}^{\st}(\Pi;q) = \sum_{(\eps,\sigma) \in \Av_{r,n}(\Pi)} q^{\st (\eps,\sigma)}?
	\]
\end{question}

We close by noting that $G_{r,n}$ is the set of elements in the wreath product $\ZZ_r \wr \symm_n$, a fact which may be useful when addressing the above questions.

\bibliographystyle{plain}
\bibliography{references}

\end{document}